\documentclass[amstex,12pt,reqno]{amsart}
\usepackage{amsmath, amssymb, amsthm}

\usepackage{mathrsfs}
\usepackage[all]{xy}
\usepackage{tikz}
\usetikzlibrary{intersections,calc,arrows.meta,patterns}

\usepackage{caption}

\newtheorem{df}{Definition}
\newtheorem{prop}{Proposition}
\newtheorem{theo}{Theorem}
\newtheorem{lem}{Lemma}

\title
[Conformal weldings]
{Conformal weldings in the Loewner equation and Weil--Petersson quasislit-disks}

\author[F. Tao]{Fei Tao} 
\address{Beijing International Center for Mathematical Research, Peking University, Beijing 100871, P. R. China} 
\email{ferrytau@pku.edu.cn}

\author[H. Wei]{Huaying Wei} 
\address{Center for applied Mathematics, Tianjin University, Tianjin 300072, P. R. China} 
\email{weihuaying2014@gmail.com}

\author[Y. Yang]{Yaosong Yang}
\address{Academy of Mathematics and Systems Science, Chinese Academy of Sciences, Beijing 100049, P. R. China}
\address{School of Mathematical Sciences, University of Chinese Academy of Sciences, Beijing 100049, P. R. China}
\email{yangyaosong@amss.ac.cn}

\makeatletter
\@namedef{subjclassname@2020}{%
\textup{2020} Mathematics Subject Classification}
\makeatother

\subjclass[2020]{30C62, 30C75, 30H25}

\keywords{Loewner differential equation, quasiarc, conformal welding, Weil--Petersson geometry}

\thanks{Research supported by 
 the National Natural Science Foundation of China (Grant No. 12271218) and the National Key R\&D Program of China (2020YFA0712800)}

\begin{document}

\captionsetup[figure]{name={Fig.}}

\begin{abstract}
A simple arc $\Gamma = \gamma(0, T]$, growing into the unit disk $\mathbb D$ from its boundary, generates a driving term $\xi$ and a conformal welding $\phi$ through the Loewner differential equation. When $\Gamma$ is the slit of a Weil--Petersson quasislit-disk $\mathbb D\setminus\Gamma$, the Loewner transform and its inverse $\Gamma \leftrightarrow \xi$ have been well understood due to Y. Wang's work \cite{Yilin19JEMS}. We investigate the maps $\Gamma \leftrightarrow \phi$ in this case, giving a description of $\Gamma$ in terms of $\phi$. 

\end{abstract}

\maketitle

\section{Introduction}
 Let $\Gamma = \gamma[0, T]$ be a simple arc with $0 \notin \Gamma$, growing into the unit disk $\mathbb D$ from its boundary $\mathbb T$. For the unit disk $\mathbb D$ slitted by a simple arc $\gamma(0, t]$, $0 < t \leq T$, by the Riemann mapping theorem, there exists a unique conformal map $f_t\colon \mathbb D\setminus\gamma(0, t] \to \mathbb D$ with the normalization $f_t(0) = 0$ and $f_t'(0) = e^t$ (by reparametrizing $\gamma(0, t]$ if necessary). Charles Loewner showed that the evolution of $f_t$ satisfies the (downward) Loewner differential equation 
\begin{equation}\label{down}
  \partial_t f_t(z) =f_t(z) \frac{\lambda(t)+f_t(z)}{\lambda(t)-f_t(z)}, \qquad f_0(z)=z,
\end{equation}
for each $z \in \mathbb D$.
Here, $\lambda(t) =f_t(\gamma(t))$ is a continuous function taking values on $\mathbb T$. 
Loewner's approach was initially developed as a tool to study extremal problems in complex analysis \cite{Loewner23}. In particular, it played an essential role in solving the Bieberbach conjecture. Later, it arose in connection with probability theory by a stochastic process introduced by Oded Schramm \cite{Schramm00} called ``Stochastic Loewner Evolution" (or SLE). 

Rather than directly working with the downward Loewner differential equation \eqref{down}, it is more natural to work with the upward Loewner differential equation
\begin{equation}\label{up}
  \partial_t g_t(z) = - g_t(z) \frac{\xi(t)+g_t(z)}{\xi(t)-g_t(z)}, \qquad g_0(z)=z,
\end{equation}
for $z \in \mathbb D$ and continuous function $\xi(t)$ taking values on $\mathbb T$. 
In this case, for each $z \in \mathbb D$, the solution to the initial value problem \eqref{up} is unique for all $t \in [0, T]$, and it is a conformal map from $\mathbb D$ into $\mathbb D$ with the normalization $g_t(0) = 0$ and $g_t'(0) = e^{-t}$. If $(g_t)_{0 \leq t \leq T}$ is the solution to \eqref{up} with $\xi(t) = \lambda(T-t)$, then it is not true that $g_t(z) = f_t^{-1}(z)$ for all $t \in [0, T]$. However, it is true that $g_T(z) = f_T^{-1}(z)$ so that $g_T(\mathbb D) = \mathbb D \setminus \Gamma$ and $g_t(\mathbb D) = \mathbb D\setminus\Gamma_t$ is a slit disk with initial point $\xi(t)$ for all $t \in [0, T]$. 

Roughly speaking, the slit $\Gamma_t$ is formed as follows. Points on $\mathbb T$ move along $\mathbb T$ towards the singularity $\xi(t)$ according to \eqref{up} until they actually hit it. At this time they leave the circle and move into $\mathbb D$. At each time $t$, there are exactly two points, located on opposite sides of $\xi(0)$, that hit $\xi(t)$ at the same time. It seems like they are ``welded together". More formally, the \emph{hitting time} for $x \in \mathbb T$ is defined by
\[
\tau(x) = \sup\{t > 0: g_s(x) \neq \xi(s) \;\text{for all}\; s \in [0, t]\}.
\]
Set $\tau(x) = t$. If $t < \infty$, it is the first time that $g_t(x) = \xi(t)$ so that the two preimages $g_t^{-1}(\xi(t))$, $x$ and $y$, satisfy $\tau(x) = \tau(y)$. This yields a welding homeomorphism $\phi$, sending $x$ to $y$; that is, it interchanges the two points that hit the singularity at the same time. 

Since $\xi$ is defined on a finite interval $[0, T]$, the welding $\phi$ will not be defined on all $\mathbb T$. Set $I = g_T^{-1}(\Gamma) \subset \mathbb T$ to be an oriented arc from $\alpha^- $ to $\alpha^+$. We actually have $g_T(x) = g_T(y)$ for $x \in I^{-} = \langle\alpha^-, \xi(0)\rangle$ and $y \in I^{+} = \langle \xi(0), \alpha^+ \rangle$ if and only if $\phi(x) = y$.

The preceding discussion leads us to a natural sequence of mappings
\[
\xi \leftrightarrow \Gamma \leftrightarrow \phi.
\]
The first pair of maps has been extensively studied in the deterministic setting and also in the stochastic setting \cite{LindSharp05,MarshallRohde05,Schramm00,Yilin19JEMS,Wong14}. The second pair has also received some attention  \cite{LindSharp05,Timminimizers23,MarshallRohde05,Yilin19JEMS} serving as a transition for the study of the former. Among all, Marshall--Rohde \cite{MarshallRohde05} and Lind \cite{LindSharp05} proved that if $\Gamma$ is the slit of a quasislit-disk, then the welding homeomorphism   $\phi$ has certain properties (see Section 2 for precise statements) and thus $\xi \in \textrm{Lip}(\frac{1}{2})$. Conversely, if $\xi \in \textrm{Lip}(\frac{1}{2})$ with $\Vert\xi\Vert_{\frac{1}{2}} < 4$, then $\phi$ is as before and $\Gamma$ is the slit of a quasislit-disk.

In this paper, in view of the importance of the arguments in \cite{LindSharp05,MarshallRohde05}, we consider the maps $\Gamma \leftrightarrow \phi$ for the Weil--Petersson case, a topic of growing interest (see \cite{Timminimizers23,TimDrivers24,Yilin19JEMS} for more discussions) due to its potential applications in string theory and other fields.

\section{Preliminaries, the main result and its motivation}

Let $\Gamma$ be a simple arc with initial point $a\in \mathbb T$ and tip $b \in \mathbb D$ contained in $\mathbb D\setminus \{0\}$. We call a subdomain $D = \mathbb D\setminus\Gamma$ of $\mathbb D$ a \emph{slit disk}. For example, the domain $D_t = \mathbb D\setminus [t, 1]$, $0 < t < 1$, is a slit disk. We now introduce two kinds of slit disks with regularity, which will be the main attention of this paper, from the perspective of quasiconformal maps. To state them precisely, we first recall some preliminary definitions and facts from the quasiconformal theory; see \cite{BookAhlfors66} for additional background. 

A quasiconformal map $f$ on $\mathbb C$ is a homeomorphism whose gradient, interpreted in the sense of distributions, belongs to $L_{\textrm{loc}}^{2}(\mathbb C)$ and whose complex dilatation $\mu = f_{\bar z}/f_z$ belongs to $L^{\infty}(\mathbb C)$ bounded by some constant $k < 1$. 

Suppose now $f$ is a quasiconformal homeomorphism of $\mathbb C$ that is conformal outside the unit disk $\mathbb D$. The image of a circular segment of the unit circle $\mathbb T$ under $f$ is called a \emph{quasiarc}, and the image of the whole unit circle $\mathbb T$ is called a \emph{quasicircle}. The collection of all quasicircles passing through three fixed points corresponds to universal Teichm\"uller space $T(1)$ and the usual metric is defined in terms of $\Vert \mu \Vert_{\infty}$. A sense-preserving homeomorphism $h$ of $\mathbb T$ is called a \emph{conformal welding} of a Jordan curve $\mathcal{C}$ if $h = \psi^{-1}\circ\varphi |_{\mathbb T} $. Here, $\varphi$ and $\psi$ are conformal maps from $\mathbb D$ onto  $\Omega$, the bounded component of the complement of $\Gamma$, and from $\mathbb D_e$ onto $\Omega_e$, where $\mathbb D_e$ and $\Omega_e$ denote the unbounded components of the complement of $\mathbb T$ and $\mathcal{C}$, respectively. By Carath\'edory extension theorem, both of them can be extended to the boundary homeomorphically. There are many weldings of $\mathcal C$ which differ from each other by compositing with M\"obius transformations of $\mathbb T$. In particular, the conformal welding of a quasicircle is exactly a quasisymmetric homeomorphism of $\mathbb T$. 

\begin{df}
  We say a (sense-preserving or sense-reversing) homeomorphism $h$ of connected closed arcs $I_1 \subset \mathbb T$ onto $I_2 \subset \mathbb T$ is quasisymmetric if there exists a constant $C \geq 1$  such that
\begin{equation}
  C^{-1} \leq \frac{\vert h(x) - h(y)\vert}{\vert h(y) - h(z)\vert} \leq C
\end{equation}
  whenever $x, y, z \in I_1$ in this order with $\vert x - y\vert = \vert y - z\vert$. 
The set of all quasisymmetric homeomorphisms on $I_1$ is denoted by $\textrm{QS}(I_1)$. 
\end{df}

Beurling--Ahlfors \cite{BeurlingAhlfors56} proved that a sense-preserving homeomorphism $h$ of $\mathbb T$ is quasisymmetric if and only if there exists some quasiconformal homeomorphism of $\mathbb D$ onto itself with boundary values $h$. 
 
If one further supposes the complex dilatation $\mu$ of a quasiconformal homeomorphism $f$ is $2$-integrable under the Poincar\'e metric of $\mathbb D$, then the image of $\mathbb T$ is called a \emph{Weil--Petersson curve} (see \cite{Bishop22,CuiIAAhomeomorphism00,TTbook}).
The collection of the Weil--Petersson curves is exactly the closure of smooth curves under the Weil--Petersson metric, and it is exactly the connected component of the identity in $T(1)$ viewed as a complex Hilbert manifold \cite{TTbook}. A Jordan curve is a Weil--Petersson curve if and only if its conformal welding $h$ belongs to the Weil--Petersson class \cite{WPI18}. 
\begin{df}
  We say a (sense-preserving or sense-reversing) homeomorphism $h$ from $I_1 \subset \mathbb T$ onto $I_2 \subset \mathbb T$ belongs to the Weil--Petersson class, denoted by $h \in \mathrm{WP}(I_1)$, if it is absolutely continuous with respect to the arc-length measure such that $\log |h'| \in H^{\frac{1}{2}}(I_1)$.
\end{df}
Here, the Sobolev space $H^{\frac{1}{2}}(I)$ on a connected closed arc $I \subset \mathbb T$ is defined to be the set of all integrable functions $u$ with finite semi-norm 
 \begin{equation}\label{12}
 \|u\|_{H^{\frac{1}{2}}(I)} = \left(\iint_{I\times I} \frac{|u(z_1)-u(z_2)|^2}{|z_1-z_2|^2}\frac{|dz_1|}{2\pi}\frac{|dz_2|}{2\pi}\right)^{\frac{1}{2}}. 
 \end{equation}
Cui \cite{CuiIAAhomeomorphism00} proved that a sense-preserving homeomorphism $h \in \textrm{WP}(\mathbb T)$ if and only if there exists a quasiconformal homeomorphism of $\mathbb D$ onto itself with boundary value $h$, whose complex dilatation $\mu$ is $2$-integrable under the Poincar\'e metric of $\mathbb D$. In particular, the barycentric extension (or called Douady--Earle extension) is a desired extension. 

Concerning the sense-preserving quasisymmetric homeomorphisms of 
$\mathbb T$ and the Sobolev space $H^{\frac{1}{2}}(\mathbb T)$, the following result of Nag--Sullivan \cite{NagSullivanHhalf95} is well-known. 
\begin{prop}\label{pull}
 A sense-preserving homeomorphism $h$ of $\mathbb T$ is quasisymmetric if and only if the composition operator $V_h: g \mapsto g\circ h$ gives an isomorphism of $H^{\frac{1}{2}}(\mathbb T)$. 
\end{prop}
With the above preparation, now we can introduce the following 
\begin{df}\label{disk}
  We say a slit disk $D = \mathbb D\setminus\Gamma$ is a quasislit-disk if there is a quasiconformal homeomorphism $f$ of $\mathbb C$ and $0 < t < 1$ such that $f(D_t) = D$ and $f(0) = 0$. If moreover its complex dilatation $\mu$ is $2$-integrable under the Poincar\'e metric of $D_t$, we say $D$ is a Weil--Petersson quasislit-disk. 
\end{df}
Concerning this definition, we recall two claims from Marshall--Rohde \cite{MarshallRohde05}:
\begin{enumerate}
\item[(i)] The requirement $f(D_t) = D$ implies $f(\mathbb D) = \mathbb D$, and the quasiconformal homeomorphism of $\mathbb D$ can be extended to $\mathbb C$ by  reflection. Therefore, in Definition \ref{disk} it suffices to consider a quasiconformal homeomorphism of $\mathbb D$. 
\item[(ii)] The slit disk $\mathbb D\setminus\Gamma$ is a quasislit-disk if and only if $\Gamma$ is a quasiarc that approaches $\mathbb T$ nontangentially. 
\end{enumerate}
The slit of the Weil--Petersson quasislit-disk appears to be related to the finite energy slit, introduced by Y.Wang \cite{Yilin19JEMS}. If the slit has finite energy then $\Gamma$ is a quasiarc that approaches $\mathbb T$ orthogonally. The converse is not true if $\Gamma$ has some corners or spirals.

Here and in what follows, we assume $\xi(0) = 1$ for simplicity. 
Recall that $\Gamma$ is a simple arc with initial point $a\in \mathbb T$ and tip $b \in \mathbb D$ contained in $\mathbb D\setminus \{0\}$. Suppose $g\colon \mathbb D \to \mathbb D\setminus\Gamma$ is the conformal map with $g(0) = 0$ and $g(1) = b$, and the points $\alpha^{\pm}$ are two preimages of $a$ so that $g^{-1}(\Gamma) = I = I^{-}\cup I^{+}$ where $I^{-} = \langle\alpha^-, 1\rangle$ and $I^{+} = \langle 1, \alpha^{+}\rangle$. The conformal welding $\phi$ interchanges $x$ and $y$ with $g(x) = g(y) \in \Gamma$, as defined in Section 1. 
Our work is motivated by
\begin{theo}[see \cite{LindSharp05,MarshallRohde05}]\label{MRL}
  The slit disk $\mathbb D\setminus\Gamma$ is a quasislit-disk if and only if the sense-reversing homeomorphism $\phi \in \textrm{QS}(I^{+})$ and there exists $M \geq 1$ such that
\begin{equation}
    M^{-1} \leq \frac{\vert\phi(x) - 1\vert}{\vert x - 1\vert} \leq M
\end{equation}
for all $x \in I^{+}$.
    \end{theo}

Let $\tau$ be the unique M\"obius transformation of $\mathbb T$ that sends $-i, 1, i$ to $\alpha^{-}, 1, \alpha^{+}$, respectively. The main result of this paper is 
\begin{theo}\label{main}
    The slit disk $\mathbb D\setminus\Gamma$ is a Weil--Petersson quasislit-disk if and only if the sense-reversing homeomorphism $\phi \in \textrm{WP}(I^{+})$ and 
 \begin{equation}\label{extra}
 \int_{\langle 1,i\rangle }\int_{\langle -i,1\rangle} \frac{\log^2|(\tau^{-1}\circ \phi\circ \tau)'(z_1)|}{|z_1-z_2|^2} |dz_1||dz_2| < \infty.
 \end{equation}
\end{theo}

Recall that in \cite{LindSharp05,MarshallRohde05} the slit $\Gamma$ of quasislit-disk was characterized in terms of the driving term $\xi$ by borrowing Theorem \ref{MRL} as a transition. Differently from that, in chordal case, it was shown that the slit $\Gamma$ with finite Loewner energy $I^{C}(\Gamma):=\frac{1}{2} \int_0^T\xi'(t)^2dt < \infty$ possesses various regular properties, without involving $\phi$ at all in the proof \cite{Yilin19JEMS}. When generalized to the loop case, these slits precisely correspond to the Weil--Petersson curves. We hope that Theorem \ref{main} could find applications to some other problems in Loewner theory, and also in the Weil--Petersson geometry.

\section{Proof of Theorem \ref{main}}
In preparation for the proof, we need the following crucial lemmas.

Recall an integrable function $u$ on a connected closed arc $I \subset \mathbb T$ is of $\textrm{BMO}(I)$ if 
\[
\Vert u \Vert_{\textrm{BMO}(I)} = \sup_{J \subset I} \frac{1}{|J|}\int_J |u(z) - u_J| |dz| < \infty,
\]
where the supremum is taken over all subarcs $J$ on $I$ and $u_J$ denotes the integral mean of $u$ over $J$. It is said that $u \in \textrm{VMO}(I)$ if moreover
\[
\lim_{|J|\to0} \frac{1}{|J|}\int_J |u(z) - u_J| |dz| = 0. 
\]
\begin{lem}\label{H}
Let $I_1$ and $I_2$ be connected closed arcs on $\mathbb T$. Let $\varphi$ be a homeomorphism from $I_1$ onto $I_2$. If $\varphi \in \text{\rm WP}(I_1)$, then $\varphi\in \text{\rm QS}(I_1)$.
\end{lem}
\begin{proof}
The fact $H^{\frac{1}{2}}(I_1)\subset\text{VMO}(I_1)$ has been asserted in \cite{WPIII21, WeiPWP22}. For convenience, we give a proof here. Set $u = \log |\varphi'|$ for simplicity. Let $I \subset I_1$ be any connected arc. Then we have 
\begin{align*}
 \frac{1}{|I|}\int_I |u(z) - u_I| |dz| &= \frac{1}{|I|}\int_I \Big|u(z) - \frac{1}{|I|}\int_I u(w) |dw| \Big| |dz|\\ &\leq \frac{1}{|I|^2} \int_I\int_I |u(z) - u(w)| |dz||dw|\\
 & \leq\Bigg(\frac{1}{|I|^2}\int_I\int_I|u(z) - u(w)|^2 |dz||dw|\Bigg)^{1/2}\\
 &\leq \Bigg(\int_I\int_I\frac{|u(z) - u(w)|^2}{|z-w|^2} |dz||dw|\Bigg)^{1/2}.
\end{align*}
This implies that $\Vert u \Vert_{\text{BMO}(I_1)}\leq \Vert u \Vert_{H^{\frac{1}{2}}(I_1)}$. If $u \in H^{\frac{1}{2}}(I_1)$, then $|u(z) - u(w)|^2/|z-w|^2$ is integrable on $I_1\times I_1$. Hence, for any $\varepsilon > 0$, there is $\delta > 0$ such that if $|I| < \delta$ then its integral over $I\times I$ is less than $\varepsilon$. This shows that $u \in \text{VMO}(I_1)$. Partyka asserted that $u \in \text{VMO}(I_1)$ implies $\varphi \in \text{QS}(I_1)$ \cite{PartykaEigenvalues98} (see also \cite{WPI18} for a different proof). This completes the proof. 
\end{proof}

Let $I_j = \left\langle a_j, b_j\right\rangle$ be connected closed arcs on $\mathbb T$, where $j = 1, 2$. Let $\tau_j$ be the M\"obius transformations sending $a_j, b_j$ to $-i, i$ respectively. Let $\varphi$ be a sense-preserving homeomorphism from $I_1$ onto $I_2$. We define $\psi = \tau_2 \circ \varphi \circ \tau_1^{-1}$ on $\left\langle -i, i \right\rangle$, and extend it to the left half circle $\left\langle i, -i \right\rangle$ by reflection, and denote it by
 \begin{equation}
 \notag \hat \psi(z)= 
 \begin{cases}
 \psi(z), & z \in \left\langle -i,i \right\rangle,\\
 -\overline{\psi(-\bar{z})}, & z \in \left\langle i, -i \right\rangle.
 \end{cases}
 \end{equation}
Set $\hat\varphi = \tau_2^{-1}\circ\hat \psi \circ \tau_1$ on $\mathbb T$. We see that $\hat\varphi|_{I_1} = \varphi$. Then we have the following results.

\begin{lem}\label{extension}
With the above notations, the following two statements hold.
\begin{enumerate}
 \item[{\rm(1)}] If $\varphi\in \text{\rm QS}(I_1)$, then $\hat \varphi \in \text{\rm QS}(\mathbb T)$.
 \item[{\rm(2)}] If $\varphi\in \text{\rm WP}(I_1)$, then $\hat \varphi \in \text{\rm WP}(\mathbb T)$.
\end{enumerate}
\end{lem}

\begin{proof}
Since $\tau_1, \tau_2$ are M\"obius transformations on $\mathbb T$ and hence bi-Lipschitz, an elementary computation shows, by the definition of quasisymmetry, that $\psi \in \text{QS}(\left\langle -i, i \right\rangle)$. By the construction of Kenig and Jerison \cite{JerisonKenigHardyAinfty82} (see also the proof of Theorem 5.8.1 in \cite{BookAstalaIwaniecMartin09}), there exists a quasiconformal self-mapping $\tilde{\psi}$ of the right semidisk such that $\tilde{\psi}|_{\left\langle -i, i \right\rangle}=\psi$. Consider the reflection of $\tilde{\psi}$ with respect to the imaginary line, still denoted by $\tilde{\psi}$. We obtain that $\tilde{\psi}$ is a quasiconformal self-mapping of $\mathbb{D}$ and $\tilde{\psi}|_{\mathbb T}=\hat{\psi}$ (e.g. Chapter 5.9 of \cite{BookAstalaIwaniecMartin09}), which yields that $\hat \psi\in \text{QS}(\mathbb T)$, and thus $\hat \varphi \in \text{QS}(\mathbb T)$. This gives a proof of part (1). 

We now prove part (2). It is easy to see that $\hat{\psi}$ is absolutely continuous on $\mathbb T$. Now we shall prove that $\log |\hat{\psi}'|\in H^{\frac{1}{2}}(\mathbb{T})$. Note that when $z \in \left\langle -i, i\right\rangle$,
\begin{equation}
 \log|\psi'(z)| = \log|\tau_2'|\circ\varphi\circ \tau_1^{-1}(z) + \log|\varphi'|\circ \tau_1^{-1}(z) + \log|(\tau_1^{-1})'(z)|.
\end{equation}
It is obvious to see that both $\log|\tau_2'|$ and $\log|(\tau_1^{-1})'|$ are in $H^{\frac{1}{2}}(\mathbb{T})$, and $\log|\tau_2'|\circ\varphi\circ \tau_1^{-1} \in H^{\frac{1}{2}}(\mathbb{T})$ using Proposition \ref{pull}, Lemma \ref{H} and part (1) of  Lemma \ref{extension}. The automorphism $\tau_1$ of $\mathbb D$ preserves the hyperbolic distance, this implies 
\[
|\tau_1(z_1) - \tau_1(z_2)|^2 = |z_1 - z_2|^2 |\tau_1'(z_1)| |\tau_1'(z_2)|,
\]
and we thus conclude by $\log|\varphi'| \in H^{\frac{1}{2}}(I_1)$ that 
\begin{align*}
 &\int_{\left\langle -i,i\right\rangle}\int_{\left\langle -i,i\right\rangle} \frac{|\log|\varphi'|\circ \tau_1^{-1}(z_1) - \log|\varphi'|\circ \tau_1^{-1}(z_2)|^2}{|z_1 - z_2|^2}|dz_1||dz_2|\\ 
 =& \int_{I_1}\int_{I_1} \frac{\left|\log|\varphi'(w_1)| - \log|\varphi'(w_2)|\right|^2}{|w_1 - w_2|^2} \frac{|w_1 - w_2|^2|\tau_1'(w_1)| |\tau_1'(w_2)|}{|\tau_1(w_1) - \tau_1(w_2)|^2} |dw_1||dw_2|\\
 =& \int_{I_1}\int_{I_1} \frac{\left|\log|\varphi'(w_1)| - \log|\varphi'(w_2)|\right|^2}{|w_1 - w_2|^2} |dw_1||dw_2| < \infty.
\end{align*}
Consequently, we obtain that $\log|\psi'| \in H^{\frac{1}{2}}(\left\langle -i,i\right\rangle)$. Notice that 
\begin{align}\label{dc}
 |\hat{\psi}'(z)| =|\psi'(-\bar z)|, \qquad z \in \left\langle i,-i\right\rangle.
\end{align}
By \eqref{dc} we have that $\log |\hat{\psi}'|\in H^{\frac{1}{2}}(\left\langle i,-i\right\rangle)$. Moreover,
\begin{align*}
 & \int_{\left\langle i,-i\right\rangle}\int_{\left\langle -i, i\right\rangle} 
\frac{|\log|\hat{\psi}'(z_1)| - \log|\hat{\psi}'(z_2)||^2}{|z_1 - z_2|^2} |dz_1| |dz_2| \\
 =&\int_{\left\langle i,-i\right\rangle}\int_{\left\langle -i, i\right\rangle} 
\frac{|\log|\psi'(-\bar{z}_1)| - \log |\psi'(z_2)||^2}{|z_1 - z_2|^2} |dz_1| |dz_2|\\
 =&\int_{\left\langle -i,i\right\rangle}\int_{\left\langle -i, i\right\rangle} 
\frac{|\log|\psi'(z_1)| - \log|\psi'(z_2)||^2}{|-\bar{z}_1 - z_2|^2} |dz_1| |dz_2|\\
 \leq& \int_{\left\langle -i,i\right\rangle}\int_{\left\langle -i, i\right\rangle} 
\frac{|\log|\psi'(z_1)| - \log|\psi'(z_2)||^2}{|z_1 - z_2|^2} |dz_1| |dz_2| < \infty.
\end{align*}
We conclude that $\log|\hat \psi'| \in H^{\frac{1}{2}}(\mathbb T)$; that is $\hat\psi \in \text{\rm WP}(\mathbb T)$. Since the set $\text{\rm WP}(\mathbb T)$ has the group structure  \cite{CuiIAAhomeomorphism00,TTbook}, we have $\hat\varphi \in \text{\rm WP}(\mathbb T)$. 
\end{proof}

\begin{lem}\label{Ma}
 Let $z_0$ be a given point in $\mathbb D$ with $|z_0| < r$. There exists a quasiconformal mapping $q$ of $\mathbb D$ such that $q = \mathrm{id}$ on the ring $r \leq |z| < 1$ and $q(z_0) \in (-r, r)$.
\end{lem}

\begin{proof}
Set $T$ to be a M\"obius transformation from $D(0, r)$, the open disk of radius $r$ centered at $0$, onto the upper-half plane $\mathbb H$ such that $T(-r) = 0$ and $T(r) = \infty$. Denote $p=T(z_0)$. We construct a quasiconformal mapping $\tilde q$ of $\mathbb H$ such that $\tilde q|_{\mathbb R} = \mathrm{id}$ and $\tilde q(p) = |p|e^{i\frac{\pi}{2}}$ locating on the imaginary line as follows: 
\begin{equation}
 \notag \tilde q(z) = 
 \begin{cases}
 |z|e^{i \left( \frac{\pi/2}{\arg p} \arg z\right)}, & \arg z \leq \arg p,
 \\
 |z| e^{\left(\frac{-\pi/2}{\pi - \arg p}(\pi - \arg z) + \pi\right)}, & \arg z \geq \arg p.
 \end{cases}
 \end{equation}
\begin{figure}
 \centering
 \includegraphics[width=10cm]{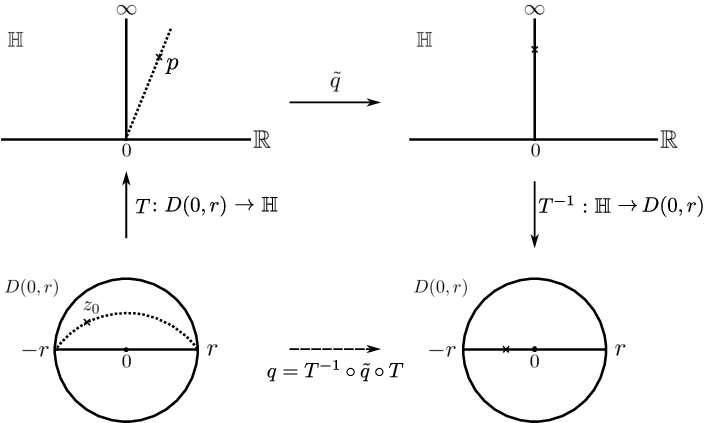}
 \caption{Illustration of the construction of the map $q\colon D(0,r)\to D(0,r)$}
 \label{fig1}
\end{figure}
Then
 \begin{equation}
 \notag q(z) = 
 \begin{cases}
 T ^{-1} \circ \tilde q \circ T (z), & 0 \leq |z| \leq r,\\
 \mathrm{id}, & r < |z| < 1
 \end{cases}
 \end{equation}
is the desired quasiconformal mapping (See Fig. 1 for a picturesque description of the construction). 
\end{proof}

\subsection{Proof of the sufficiency}

Recall that $\phi$ is a sense-reversing homeomorphism of $I = I^{+} \cup I^{-}$, and $\tau$ is the unique M\"obius transformation of $\mathbb T$ that sends $-i, 1, i$ to $\alpha^{-}, 1, \alpha^{+}$, respectively. In particular, it maps $\left\langle 1, i \right\rangle$ to $I^{+}$. Now let us use $\phi$ on $I^{+}$ to construct a sense-preserving homeomorphism of $\mathbb T$ as follows: 
 \begin{equation}
 \notag \psi (z)= 
 \begin{cases}
 z, & z \in \left\langle 1, -1 \right\rangle,\\
 \tau^{-1}\circ\phi\circ\tau(\bar z) ,& z \in \left\langle -i,1 \right\rangle,\\
 -\overline{\tau^{-1}\circ\phi\circ\tau(-z)},
 & z \in \left\langle -1, -i \right\rangle.
 \end{cases}
 \end{equation}
Notice that $\psi(z) = -\overline{\psi(-\bar z)}$ on $\left\langle i, -i \right\rangle$; that is precisely the reflection of $\psi$ with respect to the imaginary line. It is quite clear that $\psi$ is a sense-preserving homeomorphism of $\mathbb T$. With the above notations, we have the following further result. 

\begin{prop}\label{psi}
If the sense-reversing homeomorphism $\phi \in \text{\rm WP}(I^{+})$ and satisfies \eqref{extra}, then the sense-preserving homeomorphism $\psi \in \text{\rm WP}(\mathbb T)$. 
\end{prop}
\begin{proof}
It is easy to check that $\psi$ is absolutely continuous on $\mathbb T$. 
 Now we shall prove that 
 \[
\int_{\mathbb{T}}\int_{\mathbb{T}}\frac{\left|\log|\psi'(z_1)|-\log|\psi'(z_2)|\right|^2}{|z_1-z_2|^2} |dz_1||dz_2| < \infty.
 \]
 For this purpose, we divide it into several parts as follows: 
\begin{align*}
&\int_{\mathbb{T}}\int_{\mathbb{T}}\frac{\left|\log|\psi'(z_1)|-\log|\psi'(z_2)|\right|^2}{|z_1-z_2|^2} |dz_1||dz_2| \\
 =& \int_{\left\langle 1,-1\right\rangle}\int_{\left\langle 1,-1\right\rangle} \cdot
 \cdot + \int_{\left\langle -i,1\right\rangle}\int_{\left\langle -i,1\right\rangle} 
 \cdot
 \cdot+\int_{\left\langle -1,-i\right\rangle}\int_{\left\langle -1,-i\right\rangle}\cdot
 \cdot\\
 & +2\int_{\left\langle -1,-i\right\rangle}\int_{\left\langle -i,1\right\rangle} \cdot
 \cdot+2\int_{\left\langle -i,1\right\rangle}\int_{\left\langle 1,-1\right\rangle} \cdot
 \cdot+2\int_{\left\langle -1,-i\right\rangle}\int_{\left\langle 1,-1\right\rangle}\cdot
 \cdot\\
 =&: J_1+J_2+J_3+2J_4+2J_5+2J_6.
\end{align*}
Obviously, $J_1$ is bounded. The boundedness of $J_2$, $J_3$, and $J_4$ follows from the proof of Lemma \ref{extension}. The boundedness of $J_6$ follows from that of $J_5$. The boundedness of $J_5$ is from the condition \eqref{extra}. This completes the proof. 
\end{proof}

Having laid the necessary groundwork, we are now ready to give a proof of the sufficiency. 
\begin{proof}[Proof of the sufficiency]
  Suppose that $\phi \in \textrm{WP}(I^{+})$ and \eqref{extra} holds. 
It follows from Proposition \ref{psi} that the sense-preserving homeomorphism $\psi\in \mathrm{WP}(\mathbb{T})$. Then the barycentric extension of $\psi$ to $\mathbb D$, still denoted by $\psi$, is a quasiconformal mapping of $\mathbb D$ whose complex dilatation $\mu$ is 2-integrable under the Poincar\'e metric of $\mathbb D$. Obviously, there exists some $r < 1$ such that $|\psi^{-1}(\tau^{-1}(0))| < r$. We conclude by Lemma \ref{Ma} that $\beta = q(\psi^{-1}(\tau^{-1}(0))) \in (-1, 1)$. Hence, there is a unique $0 < t < 1$ such that a conformal map $h$ from $\mathbb D$ to $D_t$ exists with $h(\beta) = 0$, $h(i) = h(-i) = 1$ and $h(1) = t$. Indeed, $h$ is the composition of the maps 
\[
\frac{1-z}{1+z}, \quad c\sqrt{z^2 + 1} \quad \text{and} \quad\frac{1-z}{1+z},
\]
where $c = \left( \left(\frac{1 - \beta}{1 + \beta}\right)^2 + 1 \right)^{-\frac{1}{2}}$ and $t = \frac{1 - c}{1 + c}$. 
 \begin{figure}[htp]
 \centering
\includegraphics[width=12cm]{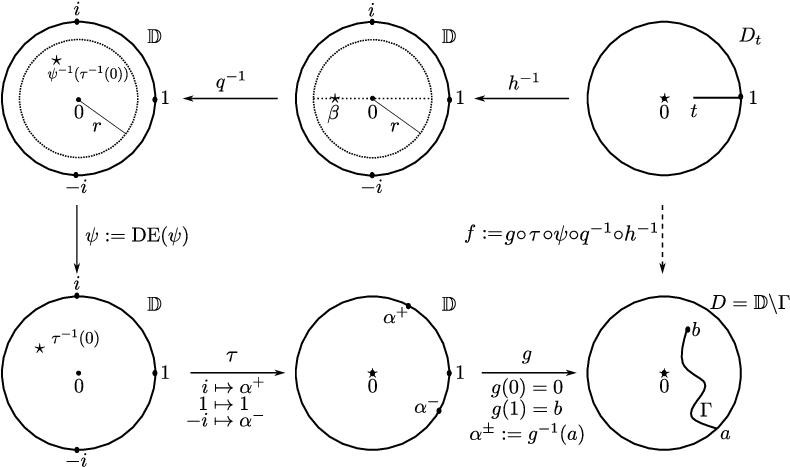}
 \caption{Illustration of the construction of the quasiconformal mapping $f\colon \mathbb{D}\to \mathbb{D}$}
 \label{fig2}
\end{figure}

We construct $f$ as $g \circ \tau \circ \psi \circ q^{-1} \circ h^{-1} $ (see Fig. \ref{fig2}). It is not hard to see that $f$ can be continuously extended to the boundary (i.e., $f$ extends continuously across $[t,1]$), mainly due to the way $\psi$ is constructed on $\mathbb T$. Precisely, for any $x \in (t, 1]$, its two preimages under $h$ are symmetric with respect to $\mathbb R$, denoted by $z \in \left\langle 1, i\right\rangle$ and $\bar z \in \left\langle -i, 1\right\rangle$. Notice that $q = \mathrm{id}$ on $\mathbb T$. Then it follows from 
\[
g\circ \tau\circ \psi(z) = g\circ \tau(z)
\]
and
\[
g\circ \tau\circ \psi(\bar z) = g\circ \tau \circ ( \tau^{-1}\circ \phi \circ \tau(z)) = g\circ \tau(z),
\]
that for any $x \in (t, 1]$ there is only one image point on the slit $\Gamma$ under the mapping $f$; namely, the free boundary arc $[t, 1)$ of $D_t$ corresponds to the free boundary arc $\Gamma$ of $D$. By the extension theorem for quasiconformal mappings (see \cite[Theorem 8.2]{BookLehtoVvirtanen73} and \cite[Proposition 4.9.9]{BookHubbard06}), $f$ can be extended to a homeomorphism of $\mathbb D$ onto $\mathbb D$. 
Thus we can define $f$ on $\mathbb D$ and it is a quasiconformal mapping of $\mathbb{D}$.

Noting that 
\[\mu_{f}=\mu_{\psi\circ q^{-1}}\circ h^{-1}\frac{\overline{(h^{-1})_z}}{(h^{-1})_z},\]
we have
\[
\iint_{D_t} |\mu_f(z)|^2\rho_{D_t}^2(z)dxdy = \iint_{\mathbb D} |\mu_{\psi\circ q^{-1}}(z)|^2\rho_{\mathbb D}^2(z)dxdy.
\]
Since $q$ is $\mathrm{id}$ on the annulus $r\leq |z| <1$, we see 
\begin{align*}
 \iint_{\mathbb D} |\mu_{\psi\circ q^{-1}}(z)|^2\rho_{\mathbb D}^2(z)dxdy =&\iint_{\mathbb{D}\setminus \mathbb{D}_r} \frac{|\mu_{ {\psi} \circ {q}^{-1} }|^2}{(1-|z|^2)^2}dxdy + \iint_{\mathbb{D}_r} \frac{|\mu_{ {\psi} \circ {q}^{-1} }|^2}{(1-|z|^2)^2}dxdy\\
 =& \iint_{\mathbb{D}\setminus \mathbb{D}_r} \frac{|\mu_{ 
 {\psi}}|^2}{(1-|z|^2)^2}dxdy + \iint_{\mathbb{D}_r} \frac{|\mu_{ {\psi} \circ {q}^{-1} }|^2}{(1-|z|^2)^2}dxdy\\
 \leq& \iint_{\mathbb{D}} \frac{|\mu_{ 
 {\psi}}|^2}{(1-|z|^2)^2}dxdy + \frac{\pi r^2}{(1-r^2)^2}\\
 <& \infty.
\end{align*}
This completes the proof of the sufficiency.
\end{proof}

\subsection{Proof of the necessity}
If $D = \mathbb D\setminus\Gamma$ is a Weil--Petersson quasislit-disk and $f(D_t) = D$. Then the complex dilatation $\mu_f$ of $f$ is $2$-integrable under the Poincar\'e metric of $D_t$; namely,
\[
\iint_{D_t}|\mu_f(z)|^2\rho^2_{D_t}(z) dxdy < \infty.
\]
\begin{figure}[htp]
 \centering
\includegraphics[width=10cm]{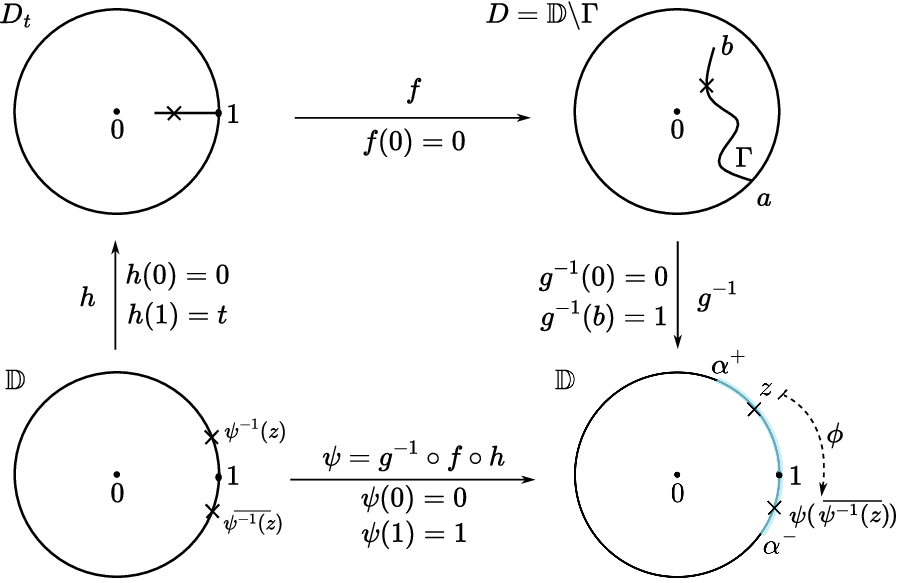}
\caption{Illustration of the construction of the conformal welding $\phi\colon I^{+}\cup I^{-}\to I^{+}\cup I^{-}$}
\end{figure}
Consider the conformal map $h\colon \mathbb D \to D_t$ with $h(0) = 0$, $h(1) = t$, $h(\pm i)=1$ (as constructed in the proof of the sufficiency and set $\beta=0$). Then $\psi = g^{-1}\circ f\circ h$ is a quasiconformal self-mapping of $\mathbb D$ fixing $0$ and $1$, and we have 
\begin{align*}
\iint_{\mathbb{D}}|\mu_\psi(z)|^2\rho^2_{\mathbb{D}}(z) dxdy=\iint_{\mathbb{D}}|\mu_{f}\circ h(z)|^2\rho^2_{\mathbb{D}} (z)dxdy=\iint_{D_t}|\mu_f(z)|^2\rho^2_{D_t}(z) dxdy < \infty.
\end{align*}
Therefore, the extension of $\psi$ to $\mathbb T$, still denoted by $\psi$, belongs to $\textrm{WP}(\mathbb T)$. It is clear to check that the conformal welding of $D$ is exactly $\phi(z) = \psi(\overline{\psi^{-1}(z)})$ on $I$ (see Fig. 3). Set $\iota(z)=\bar{z}$ so that $\phi(z) = \psi\circ \iota \circ \psi^{-1}(z)$. By an easy computation we see $\iota\circ\psi\circ\iota\in \mathrm{WP}(\mathbb{T})$, then $\iota\circ \phi(z) =\iota\circ\psi\circ \iota \circ \psi^{-1}(z)\in \mathrm{WP}(\mathbb{T})$. We thus have $\phi \in \textrm{WP}(\mathbb{T})$, and in particular $\phi \in \textrm{WP}(I^{+})$. 

It remains to show that $\phi$ satisfies \eqref{extra}. 
Notice that $\psi^{-1}\circ \tau\in \mathrm{WP}(\mathbb{T})$ and $\psi^{-1}\circ \tau$ maps $\langle 1,i\rangle$ onto itself. 
One can extend $\psi^{-1}\circ \tau$ on $\langle 1,i\rangle$ to the whole unit circle $\mathbb T$ by the following reflection: 
\begin{equation}
 \notag \Psi (z)= 
 \begin{cases}
 \psi^{-1}\circ \tau(z), & z \in \left\langle 1, i \right\rangle,\\
 \iota\circ\psi^{-1}\circ \tau\circ\iota(z),
 & z \in \left\langle -i, 1 \right\rangle,\\
 -\overline{\psi^{-1}\circ\tau(-\bar{z})}, & z \in \left\langle i, -1 \right\rangle,\\
 -\psi^{-1}\circ\tau(-z), & z \in \left\langle -1,-i \right\rangle.
 \end{cases}
 \end{equation}
A direct computation using \eqref{12} shows that $\log |\Psi'| \in H^{\frac{1}{2}}(\mathbb T)$ so that $\Psi\in \mathrm{WP}(\mathbb{T})$. Therefore, $(\iota\circ\psi^{-1}\circ \tau\circ\iota)^{-1}\circ \Psi\in \mathrm{WP}(\mathbb{T})$. Here,
 \begin{equation}
 \notag (\iota\circ\psi^{-1}\circ \tau\circ\iota)^{-1}\circ\Psi (z)= 
 \begin{cases}
 \iota\circ \tau^{-1}\circ\phi\circ \tau(z), & z \in \left\langle 1, i \right\rangle,\\
 z,
 & z \in \left\langle -i, 1 \right\rangle.
 \end{cases}
 \end{equation}
 Notice that $|(\iota\circ \tau^{-1}\circ\phi\circ \tau(z))'|=|(\tau^{-1}\circ\phi\circ \tau(z))'|$. We conclude that \eqref{extra} holds.

\medskip

\noindent\textbf{Acknowledgement} We would like to thank Katsuhiko Matsuzaki for having shared his idea leading to  Lemma \ref{Ma}.

\bibliographystyle{alpha}

\end{document}